\documentclass[12pt]{amsart}  
\usepackage{amsmath,amsthm,amssymb}
\newdimen\slantmathcorr
\def\oversl#1{
\setbox0=\hbox{$#1$}
\slantmathcorr=\wd0
\hskip 0.2\slantmathcorr \overline{\hbox to 0.8\wd0{
\vphantom{\hbox{$#1$}}}}
\hskip-\wd0\hbox{$#1$}
}

\newtheorem{theorem}{Theorem}[section]
\newtheorem{Lemma}[theorem]{Lemma}

\newtheorem{Corollary}[theorem]{Corollary}
\newtheorem{Definition}[theorem]{Definition}
\newtheorem{Example}[theorem]{Example}

\begin{document}
\title{$GH$-stability and spectral decomposition for group actions}      
\author{Abdul Gaffar Khan, Pramod Das, Tarun Das} 
\maketitle
\begin{center}
Department of Mathematics, Faculty of Mathematical Sciences, 
\\
University of Delhi, New Delhi-110007. 
\\
gaffarkhan18@gmail.com, pramod.math.ju@gmail.com, tarukd@gmail.com
\end{center}    
 
\begin{abstract} 
We study expansivity and the shadowing property for finitely generated group actions on metric spaces. We consider the projecting and lifting problems for actions having these properties. We prove that every expansive action with the shadowing property is strongly $GH$-stable (Gromov-Hausdorff stable). Finally, we introduce sequential shadowing property for finitely generated group actions on metric spaces and show that such shadowing is strong enough to imply the spectral decomposition property.       
\end{abstract}
\maketitle 


     
\section{Introduction} 

The theory of dynamical systems deal with dynamical properties of group actions. In more restricted case, a system is called discrete or continuous depending on whether the group under consideration is $\mathbb{Z}$ or $\mathbb{R}$. Although the shadowing property was originated from the work \cite{A} of Anosov in differentiable dynamics, it is one of the extensively studied properties for homeomorphisms or continuous maps in topological dynamics. The underlying idea is to trace each approximate orbit by an actual orbit so that the behaviour of the actual orbit reflects into the approximate one. This property plays significant role in the numerical investigation of various modern dynamical systems occurred in different areas of Applied Mathematics, Physics, Engineering and so on. Though the shadowing is very strong property, mathematicians have found several important dynamical systems satisfying the same. For instance, the tent map on the unit interval, the period doubling map of the circle and the shift map of the shift space on two symbols satisfy this property. Another important dynamical property in topological dynamics is widely known as expansivity. Although the expansive property of symbolic flows was recognized earlier, the first formal introduction of this property for arbitrary homeomorphisms was done in \cite{U1} by Utz. The book \cite{A1} can impart a great deal of knowledge to the reader on expansivity and the shadowing property for homeomorphisms and continuous maps. 
\medskip

In 1970, P. Walters proved \cite{U2} that Anosov diffeomorphisms are topologically stable, a property in the presence of which continuous perturbations does not affect the nature of the system. In 1971, Z. Nitecki proved [\cite{N}, p. 108] that Anosov diffeomorphisms are expansive. In 1975, Bowen proved [\cite{B}, p. 74] that Anosov diffeomorphisms also satisfy the shadowing property. These facts naturally guided Walters to prove \cite{U3} that expansive homeomorphisms with the shadowing property on compact metric spaces are topologically stable. These two results regarding topological stability are popularly known as Walters stability theorems. In \cite{LM}, authors introduced shadowing and topological stability for Borel measures and proved that an expansive measure \cite{M} with shadowing is topologically stable. In \cite{A3}, authors replaced the usual $C^{0}$-distance with the $C^{0}$-Gromov-Hausdorff distance to define the $GH$-stability for homeomorphisms on a compact metric space in the class of all homeomorphisms on possibly different compact metric space. They proved that expansive homeomorphisms with the shadowing property on compact metric space are $GH$-stable. Smale proved \cite{S1} that the non-wandering set of Anosov diffeomorphisms can be decomposed into union of finitely many disjoint closed invariant sets on each of which the diffeomorphism is transitive. Aoki extended this result \cite{A21} to homeomorphisms on compact metric spaces satisfying expansivity and the shadowing property. These two results are popularly known as Smale's spectral decomposition theorems.  
\medskip

Since the dynamics of a homeomorphism on a compact metric space can be seen as the dynamics of the corresponding induced $\mathbb{Z}$-action, it is natural to try to understand the dynamics of finitely generated group actions. One of the most interesting study on expansive actions was published by Hurder in \cite{H1}. Although $\mathbb{S}^{1}$ does not admit expansive homeomorphism, he was able to provide example of expansive action on $\mathbb{S}^1$. Recently, authors of \cite{BDS} studied properties of expansive group actions. In \cite{P1}, authors introduced the shadowing property for $\mathbb{Z}^{p}$-action and obtained sufficient conditions for a linear $\mathbb{Z}^{p}$-action on $\mathbb{C}^{m}$ to have this property. In \cite{O2}, authors extended the shadowing property to finitely generated group actions. In \cite{C3}, authors proved that finitely generated group actions with expansivity and the shadowing property are topologically stable. In \cite{DKY}, authors proved a measurable version of this result for finitely generated group actions. Since spectral decomposition \cite{A1} is a fundamental property of a homeomorphism with expansivity and the shadowing property, it must be interesting to investigate the validity of such a decomposition for finitely generated group actions. In this regard, the study of chain recurrent set is very important. 
\medskip

In this paper, we continue exploring the dynamics of finitely generated group actions on compact metric spaces satisfying expansivity, the shadowing property and sequential shadowing property. In section 2, we recall preliminaries on finitely generated group actions required for subsequent sections. In section 3, we study the chain recurrent set for actions satisfying the shadowing property and characterize expansivity in terms of existence of generators. We discuss the projecting and lifting problems for actions satisfying expansivity as well as the shadowing property. In section 4, we introduce the property of strong $GH$-stability and $GH$-stability for finitely generated group actions and prove that expansive actions with the shadowing property are strongly $GH$-stable and hence, $GH$-stable. In the final section, we introduce the property of sequential shadowing for finitely generated group actions and prove that the phase space of any action with the sequential shadowing property can be decomposed into finitely many closed invariant sets on each of which the action is transitive.   
    
\section{Preliminaries}   
Throughout the paper, $G$ denotes a finitely generated group with finite symmetric generating set $S=\lbrace s_{i}:1\leq i\leq n\rbrace$, $O(G)$ denotes the cardinality of $G$ and $X$, $Y$, $Z$ denote metric spaces. For $k\in\mathbb{N}$, set $G_{k}=\lbrace g\in G:L(g)\leq k\rbrace$, where $L(g)$ denotes the length of $g$ with respect to the generating set $S$. We assume that the length of the identity element $e$ is $0$. Note that, if $G$ is a finite group, then there exists the least $k\in \mathbb{N}$ such that $G_{k+i}= G$, for all $i\geq 1$ and if $G$ is a countably infinite group then $G_{k}\setminus G_{k-1}\neq\phi$, for all $k\in \mathbb{N}$.   
\medskip

A map $\Phi:G\times X\rightarrow X$ is said to be a continuous action of $G$ on $X$ if the following conditions hold.  
\begin{enumerate}
\item [(i)] $\Phi_g=\Phi(g,.)$ is a homeomorphism, for every $g\in G$. 
\item [(ii)] $\Phi_{e}(x)=x$, for all $x\in X$ , where $e$ is the identity element of the group $G$.
\item [(iii)] $\Phi_{g_1g_2}(x)=\Phi_{g_1}(\Phi_{g_2}(x))$, for all $x\in X$ and for all $g_1,g_2\in G$.
\end{enumerate}
\medskip

An action $\Phi$ is said to be uniformly continuous if for each $g\in G$, $\Phi_g$ is uniform equivalence (i.e. both $\Phi_g$ and $\Phi_{g^{-1}}$ are uniformly continuous). We denote the set of all uniformly continuous actions of $G$ on $X$ by $Act(G,X)$. If the phase space is compact, then continuous actions and uniformly continuous actions are the same. An action $\Phi$ is said to be commutative if $\Phi_{s}\circ \Phi_{s'}= \Phi_{s'}\circ \Phi_{s}$ for all $s, s'\in S$. A subset $W\subset X$ is said to be $\Phi$-invariant if $\Phi_g(W)\subset W$, for all $g\in G$.  
\medskip

An action $\Phi\in Act(G,X)$ is said to be expansive, if there exists a constant $\mathfrak{e}>0$ such that for any pair of distinct points $x,y\in X$, there exists $g\in G$ such that $d(\Phi_g(x),\Phi_g(y))>\mathfrak{e}$. Such a constant $\mathfrak{e}$ is said to be an expansive constant for $\Phi$.
\medskip  

Let $\Phi\in Act(G,X)$ and let $\delta>0$. A map $f:G\rightarrow X$ is said to be a $\delta$-pseudo orbit for $\Phi$ (with respect to the generating set $S$) if $d(\Phi_s(f(g)),f(sg))<\delta$, for all $s\in S$, $g\in G$. On the other hand, a map $f:G\rightarrow X$ is said to be $\delta$-traced by some point $x\in X$ if $d(\Phi_g(x),f(g))<\delta$, for all $g\in G$. An action $\Phi$ is said to have shadowing property (with respect to the generating set $S$) if for every $\epsilon>0$ there exists $\delta>0$ such that every $\delta$-pseudo orbit is $\epsilon$-traced by some point in $X$ \cite{O2}.   
\medskip 

Two actions $\Phi\in Act(G,X)$ and $\Psi\in Act(G,X)$ are said to be uniformly conjugate if there exists a uniform equivalence $h:X\rightarrow Y$ such that $h\circ \Phi_g=\Psi_g\circ h$, for all $g\in G$. The map $h$ is called a uniform conjugacy between $\Phi$ and $\Psi$. 
\medskip

The properties of finitely generated group actions which are invariant under uniform conjugacy are called uniform dynamical property. It is known that the shadowing property and expansivity are uniform dynamical property. 
\medskip

A point $x\in X$ is said to be a non-wandering point for $\Phi\in Act(G, X)$ if for every non-empty open neighbourhood $U$ of $x$, there exists a non-identity element $g\in G$ such that $\Phi_{g}(U)\cap U\neq \phi$. The set of all non-wandering points is denoted by $\Omega(\Phi)$. An action $\Phi\in Act(G, X)$ is said to be transitive if for every pair of non-empty open sets $U$ and $V$ of $X$, there exists a non-identity element $g\in G$ such that $\Phi_{g}(U)\cap V\neq \phi$.
\medskip

A continuous onto map $\pi : Y\rightarrow X$ is said to be locally isometric covering map if for each $x\in X$, there exists a neighbourhood $U(x)$ of $x$ such that $\pi^{-1}(U(x)) = \cup_{\alpha}U_{\alpha}$, where $\lbrace U_{\alpha}\rbrace$ is a pairwise disjoint family of open sets and $\pi|_{U_{\alpha}} : U_{\alpha}\rightarrow U(x)$ is an isometry for each $\alpha$. 
\medskip

An isometry between $(X,d_X)$ and $(Y,d_Y)$ is an onto map $i : X\rightarrow Y$ satisfying $d_{Y}(i(x_{1}), i(x_{2}))$ $= d_{X}(x_{1}, x_{2})$, for all $x_{1}, x_{2}\in X$. We say that $X$ and $Y$ are isometric if there exists an isometry between them. 
\medskip

For the metric space $(X,d_X)$, define a bounded metric $d^{X}(x, y)=$ min$\lbrace d_{X}(x, y), 1\rbrace$. Note that, if $\Phi\in Act(G, (X, d_{X}))$ then $\Phi\in Act(G, (X, d^{X}))$. For $A, B\subset X$, we define $d^{X}(A, B) = $inf$\lbrace d^{X}(a, b) \mid a\in A, b\in B\rbrace$. If $A = \lbrace a\rbrace$ then we write $'a'$ only. The Hausdorff distance between two subsets $A,B\subset X$ is given by $d_{H}^{X}(A, B) =$ max $\lbrace$ sup$_{a\in A}d^{X}(a, B)$, sup$_{b\in B}d^{X}(A, b)\rbrace$.  
\medskip

For $\delta > 0$, a non-necessarily continuous map $i: X\rightarrow Y$ between $X$ and $Y$ is said to be a $\delta$-isometry if max$\lbrace d^{Y}_{H}(i(X), Y)$, sup$_{x_{1},x_{2}\in X}|d^{Y}(i(x_{1}), i(x_{2})) - d^{X}(x_{1}, x_{2})|\rbrace < \delta$.  
    


\section{Expansivity and the Shadowing Property} 

In this section, the purpose is to prove the validity of some interesting discrete dynamical results on expansivity and the shadowing property for finitely generated group actions. 
\medskip

Let $x,y\in X$ and $\delta>0$. We say that a map $f:G\rightarrow X$ is a $\delta$-pseudo orbit of $\Phi\in Act(G,X)$ from $x$ to $y$ if there exists distinct $h,k\in G$ such that $f(h)=x$ and $f(k)=y$. Then, $x$ is said to be $\delta$-related to $y$ (written as $x R^S_{\delta} y$) if there exists a $\delta$-pseudo orbit of $\Phi$ from $x$ to $y$. If $x R^S_{\delta} y$ for each $\delta>0$, then we say that $x$ is related to $y$ (written $x R^S y$). The set $CR^S(\Phi)=\lbrace x\in X\mid x R^S x\rbrace$ is called the chain recurrent set of $\Phi$. An action $\Phi\in Act(G,X)$ is said to be chain transitive with respect to $S$ if $x R^S y$ for all $x,y\in X$.  
\medskip

Because of the following result we denote the chain recurrent set by $CR(\Phi)$.   

\begin{theorem}
Let $\Phi\in Act(G, X)$. Then, chain transitivity of $\Phi$ and the chain recurrent set of $\Phi$ are independent of the generating set.    
\label{4.7}
\end{theorem}
\begin{proof}
Let $S$ and $S'$ be two generating sets of $G$. Since any two word norms are bi-Lipschitz equivalent, we can choose a constant $C \geq 1$ such that $\frac{|g|_{S'}}{C}\leq |g|_{S} \leq C|g|_{S'}$ for all $g\in G$. Let $\epsilon > 0$ and $0 < \delta < \frac{\epsilon}{C}$ be such that  $d(\Phi_{g}(a), \Phi_{g}(b)) < \frac{\epsilon}{C}$, whenever $d(a,b) < \delta$ for all $a,b\in X$ and $g\in G$ satisfying $|g|_{S'}\leq C$. Let $f$ be a $\delta$-pseudo orbit of $\Phi$ from $x$ to $y$ with respect to $S'$ i.e. $d(\Phi_{s'}(f(g)), f(s'g)) < \delta$ for all $s'\in S$, $g\in G$. 
\medskip

Then, for any $s'_{1},...,s'_{C}\in S'$, $g\in G$
\begin{align*}
d(\Phi_{s'_{C}...s'_{1}}(f(g)), f(s'_{C}...s'_{1}g)) &\leq d(\Phi_{s'_{C}...s'_{2}}(\Phi_{s'_{1}}f(g)), \Phi_{s'_{C}...s'_{2}}(f(s'_{1}g)))\\ 
&+ d(\Phi_{s'_{C}...s'_{2}}(f(s'_{1}g)), \Phi_{s'_{C}...s'_{3}}(f(s'_{2}s'_{1}g)))\\ 
&+...+ d(\Phi_{s'_{C}}(\Phi_{s'_{C-1}}f(s'_{C-2}...s'_{1}g)), \Phi_{s'_{C}}f(s'_{C-1}...s'_{1}g))\\ 
&+ d(\Phi_{s'_{C}}f(s'_{C-1}...s'_{1}g), f(s'_{C}...s'_{1}g)) \\
&<  \frac{\epsilon}{C} + \frac{\epsilon}{C} + ...\delta < \epsilon 
\end{align*}

Thus, $d(\Phi_{h}(f(g), f(hg)) < \epsilon$ for all $g\in G$, $h\in G$ satisfying $|h|_{S'} \leq C$. Since any $s\in S$ satisfies $|s|_{S'}\leq C$, $f$ is an $\epsilon$-pseudo orbit of $\Phi$ from $x$ to $y$ with respect to $S$. So for any pair $x,y\in X$, $x R^{S'} y$ implies $x R^S y$. Therefore, if $\Phi$ is chain transitive with respect to $S'$ then it is chain transitive with respect to $S$ and we conclude that chain transitivity of $\Phi$ is independent of the generating set. 
\medskip

Further, the argument implies that $CR^{S'}(\Phi) \subset CR^{S}(\Phi)$ and in a similar manner, the converse also holds. Therefore, we conclude that the chain recurrent set is independent of the generating set.     
\end{proof}    

\begin{theorem} 
Let $\Phi\in Act(G, X)$ and $\Psi\in Act(G, Y)$ be conjugate. Then, $\Phi$ is chain transitive if and only if $\Psi$ is chain transitive. 
\label{4.8} 
\end{theorem}
\begin{proof}
Let $h:X\rightarrow Y$ be a conjugacy between $\Phi$ and $\Psi$. Suppose that $\Phi$ is chain transitive. Let $\epsilon>0$ and $y_{1}, y_{2}\in Y$. Then there exists $x_{1}, x_{2}\in X$ such that $h(x_{1}) = y_{1}$, $h(x_{2}) = y_{2}$. Corresponds to $\epsilon$, choose $0 < \delta < \epsilon $ by the uniform continuity of $h$. By chain transitivity of $\Phi$, there exists a $\delta$-pseudo orbit $f:G\rightarrow X$ for $\Phi$ from $x_{1}$ to $x_{2}$ satisfying $f(g_{1}) = x_{1}$ and $f(g_{2}) = x_{2}$ for some $g_1,g_2\in G$. Then, the map $hf:G\rightarrow Y$ satisfies $p(\Psi_{s}(hf(g)), hf(sg)) = p(h\Phi_{s}(f(g)), h(f(sg))) < \epsilon$ for all $s\in S$, $g\in G$. Thus, $hf$ is $\epsilon$-pseudo orbit for $\Psi$ from $y_{1}$ to $y_{2}$. Therefore, we conclude that $\Psi$ is chain transitive. One can prove the converse similarly. 
\end{proof} 
 
\begin{theorem}
If $\Phi\in Act(G,X)$, then the following statements are true 
\begin{enumerate}
\item[(i)] $CR(\Phi)$ is closed.
\item[(ii)] If $\Phi$ is commutative, then $CR(\Phi)$ is $\Phi$-invariant.    
\end{enumerate} 
\label{4.9}
\end{theorem}
\begin{proof} 
(i) Let $\epsilon>0$ and $0<\delta<\frac{\epsilon}{3}$ be such that $d(a,b)<\delta$ implies $d(\Phi_{s}(a),\Phi_{s}(b))<\frac{\epsilon}{3}$, for all $s\in S$. Let $\lbrace x_{n}\rbrace$ be a sequence in $CR(\Phi)$ such that $x_{n}$ converges to $x$. 
Choose $k>0$ such that $d(x,x_{n})<\delta$ for all $n\geq k$. Since $x_{k}\in CR(\Phi)$, there exists $\delta$-pseudo orbit $f_{1}:G\rightarrow X$ from $x_{k}$ to $x_{k}$ satisfying $f_{1}(h_{1})=x_{k}=f_{1}(h_{2})$ for distinct $h_{1},h_{2}\in G$. Define $f_2:G\rightarrow X$ by $f_{2}(g)=f_{1}(g)$ if $g\in G\setminus \lbrace h_{1},h_{2}\rbrace$ and $f_{2}(g)=x$ if otherwise. We claim that $f_{2}$ is an $\epsilon$-pseudo orbit from $x$ to $x$. 
\begin{enumerate}
\item[Case 1] Let $sg,g\in G\setminus \lbrace h_{1},h_{2}\rbrace$. Then $d(\Phi_{s}(f_{2}(g)),f_2(sg))=d(\Phi_{s}(f_{1}(g)),f_1(sg))< \delta <\epsilon$ for all $s\in S$. 

\item[Case 2] Let $sg\in G\setminus \lbrace h_{1},h_{2}\rbrace$, $g\in \lbrace h_{1},h_{2}\rbrace$. 
We have $d(\Phi_{s}(f_{1}(g),f_1(sg))<\delta$ for all $s\in S$ 
and $d(\Phi_{s}(f_{1}(g)),\Phi_{s}(x))=d(\Phi_{s}(x_{k}),\Phi_{s}(x))<\frac{\epsilon}{3}$ by the uniform continuity of $\Phi_s$. 
Thus, we have $d(f_{2}(sg), \Phi_{s}(f_{2}(g)))= d(f_{1}(sg),\Phi_{s}(x))\leq d(f_{1}(sg),\Phi_{s}(f_{1}(g))+$ $d(\Phi_{s}(f_{1}(g)), \Phi_{s}(x))<\epsilon$. 

\item[Case 3] Let $sg\in \lbrace h_{1},h_{2}\rbrace$, $g\in G\setminus \lbrace h_{1},h_{2}\rbrace$.
Then, $d(\Phi_{s}(f_{2}(g)),f_2(sg))=d(\Phi_{s}(f_{1}(g),x)\leq d(x,x_{k})+d(x_{k},\Phi_{s}(f_{1}(g)))=d(x,x_{k})+d(f_{1}(sg),\Phi_{s}(f_{1}(g)))<\epsilon$. 

\item[Case 4] Let $sg,g\in \lbrace h_{1},h_{2}\rbrace$. Then, $d(f_{2}(sg),\Phi_{s}(f_{2}(g)))=d(x,\Phi_{s}(x))\leq d(x,x_{k})+d(x_{k}, \Phi_{s}(f_{1}(g)))+d(\Phi_{s}(f_{1}(g)),\Phi_{s}(x)) =
d(x,x_{k})+d(f_{1}(sg),\Phi_{s}(f_{1}(g))) + d(\Phi_{s}(x_k), 
\\
\Phi_{s}(x))<\epsilon$. This completes a proof.    
\end{enumerate}
\medskip

(ii) Let us fix $p\in S$ and $y\in \Phi_{p}(CR(\Phi))$, then there exists $x\in CR(\Phi)$ such that $y=\Phi_p(x)$. Let $\epsilon>0$ be given and $\delta>0$ such that $d(a,b)<\delta$ implies $d(\Phi_{s}(a),\Phi_{s}(b))<\epsilon$, for all $s\in S$. Since $x\in CR(\Phi)$, there exists a $\delta$-pseudo orbit $f:G\rightarrow X$ from $x$ to $x$. Note that $\Phi_{p}f:G\rightarrow X$ is an $\epsilon$-pseudo orbit from $\Phi_{p}(x)$ to $\Phi_{p}(x)$ and hence $\Phi_p(CR(\Phi))\subset CR(\Phi)$. Since the generating set is symmetric, $CR(\Phi)\subset \Phi_{p}(CR(\Phi))$ for all $p\in S$ and hence $\Phi\in Act(G, CR(\Phi))$. 
\end{proof} 

\begin{theorem}
Let $X$ be a compact metric space. If $\Phi\in Act(G,X)$ is a commutative expansive action and $\Phi|_{CR(\Phi)}$ has the shadowing property, then $CR(\Phi)$ is isolated.     
\label{4.10}
\end{theorem}
\begin{proof} 
Let $\mathfrak{c}$ be an expansive constant for $\Phi$. For $0 <\beta < \frac{\mathfrak{c}}{2}$, choose $\alpha > 0$ by the shadowing property of $\Phi|_{CR(\Phi)}$. Further, choose $0 < \gamma < min\lbrace \frac{\alpha}{2}, \frac{\mathfrak{c}}{2} \rbrace$ such that $d(x,y) < \gamma$ implies that $d(\Phi_{s}(x), \Phi_{s}(y)) < \frac{\alpha}{2}$, for all $s\in S$. Set $U = \lbrace y\in X\mid d(y, CR(\Phi)) \leq \gamma \rbrace$ and $\cap_{g\in G}\Phi_{g}(U) = U_{\Phi}$. Clearly, $U$ is compact being a closed subset of a compact metric space. It is also easy to see that $CR(\Phi)\subset U_{\Phi}$. For the reverse containment, choose $y\in U_{\Phi}$. Since $\Phi_{g}(y)\in U$ for all $g\in G$, there exists $x_{g}\in CR(\Phi)$ such that $ d(\Phi_{g}(y) , x_{g}) \leq \gamma$, for all $g\in G$. Thus $d(\Phi_{s}(x_{g}), x_{sg}) \leq d(\Phi_{s}(x_{g}), \Phi_{sg}(y)) + d(\Phi_{sg}(y), x_{sg}) < \frac{\alpha}{2} + \gamma < \alpha$, for all $g\in G$ and all $s\in S$. Thus, $f : G\rightarrow X$ defined by $f(g) = x_{g}$ forms a $\alpha$-pseudo orbit for $\Phi|_{CR(\Phi)}$. Let $x\in CR(\Phi)$ be a $\beta$-tracing point for $f$ but this implies that $d(\Phi_{g}(x), \Phi_{g}(y))\leq d(\Phi_{g}(x), x_{g}) + d(x_{g}, \Phi_{g}(y)) < \beta + \gamma < \mathfrak{c}$. By the expansivity of $\Phi$, we get that $x = y$ and therefore $U_{\Phi} \subset CR(\Phi)$. Hence, $CR(\Phi)$ is isolated. 
\end{proof}   

We now extend the concept of generators \cite{K1} for homeomorphisms to group actions and characterize expansivity in terms of existence of generators.    
\begin{Definition} 
A finite open cover $\Lambda= \lbrace A_{1}, A_{2}, . . ., A_{n} \rbrace$ of $X$ is said to be  
\\
(i) generator for $\Phi\in Act(G,X)$ if for any family $\lbrace U_{g} \rbrace _{g \in G}$ of members of $\Lambda$, $\cap_{g \in  G} \Phi_{g^{-1}}(\oversl{U_{g}})$ contains at most one point. 
\\
(ii) weak generator for $\Phi$ if for any family $\lbrace U_{g} \rbrace _{g \in G}$ of members of $\Lambda$, $\cap_{g \in G} \Phi_{g^{-1}}(U_{g})$ contains at most one point. 
\label{3.3}
\end{Definition} 

\begin{theorem} 
Let $X$ be a compact metric space and $\Phi \in Act(G,X)$. Then, the following statements are equivalent 
\begin{enumerate}
\item[(i)] $\Phi$ is expansive.
\item[(ii)] $\Phi$ admits a generator.
\item[(iii)] $\Phi$ admits a weak generator.
\end{enumerate}
\label{3.4}  
\end{theorem}
\begin{proof}
$(i)\Rightarrow(ii)$ 
Suppose that $\Phi$ is expansive with expansive constant $\mathfrak{c}>0$. Let $\beta$ be a finite open cover of $X$ consisting of open balls of radius $\mathfrak{c} / 2$ and $\lbrace B_{g}\rbrace$ be an arbitrary family of members from $\beta$. If $x ,y \in \cap_{g \in G}(\Phi_{g^{-1}}(\oversl{B_{g}}))$, then $\Phi_{g}(x) , \Phi_{g}(y) \in \oversl{B_{g}}$, for all $g \in G$ or $d(\Phi_{g}(x) , \Phi_{g}(y)) \leq \mathfrak{c}$, for all $g \in G$. By expansivity of $\Phi$, we must have $x=y$. Hence, $\beta$ is a generator for $\Phi$. 
\medskip

$(ii)\Rightarrow(iii)$ Follows from the definition.
\medskip

$(iii)\Rightarrow(i)$
Let $\alpha$ be a weak generator for $\Phi$ and let $\mathfrak{c}$ be the Lebesgue number for $\alpha$.  
Our claim is that $\Phi$ is expansive with expansive constant $\mathfrak{c}$. Let $x,y\in X$ such that $d(\Phi_{g}(x) , \Phi_{g}(y)) \leq \mathfrak{c}$ for all $g\in G$. Since $\mathfrak{c}$ is the Lebesgue number for $\alpha$, there exists $A_{g'} \in \alpha$ such that $\Phi_{g}(x) ,\Phi_{g}(y)\in A_{g^{'}}$ for each $g \in G$. Consider $\lbrace A_{g} \rbrace_{g\in G}$ such that $A_{g}= A_{g'}$ for all $g \in G$. Therefore, $x , y \in \cap_{g \in G}(\Phi_{g^{-1}}(A_{g}))$ which is a contradiction to our assumption. Hence $x = y$ and thus $\Phi$ is expansive with expansive constant $\mathfrak{c}$.  
\end{proof}

The following example justifies the fact that Theorem \ref{3.4} does not hold for non-compact spaces. 
\begin{Example} Let $X=\mathbb{Z}$ be the set of integers equipped with the discrete metric and $G=(\mathbb{Z}^{2}, +)$ with symmetric generating set $\lbrace e_{1}, e_{2}, e_{3}, e_{4}\rbrace$, where $e_{1}= (1, 0)$, $e_{2}= (0, 1)$, $e_3= (-1, 0)$ and $e_4= (0, -1)$. Let $\Phi$ be the action generated by $\Phi_{e_{1}}(x)= 1-x = \Phi_{e_{2}}(x)$ for all $x\in X$. One can observe that the action is expansive. Further, the action does not admit a generator. If possible, suppose $\alpha = \lbrace V_{1}, V_{2}, . . . V_{n} \rbrace$ is a generator. Note that $\oversl{V_{j}} = V_{j}$ for all $1\leq j\leq n$ and without loss of generality, we can assume that $V_{1}$ is infinite. Choose $U_{g} = V_{1}$ to show that, if $x\in V_{1}$ then $(1-x)\notin V_{1}$. Similarly, we can show that for every $2\leq j\leq n$ and every pair of distinct points $x,y\in V_{1}$, either $(1-x)\notin V_{j}$ or $(1-y)\notin V_{j}$, which is not possible. Thus $\Phi$ does not admit a generator. 
\label{3.6}
\end{Example}

\begin{Definition}
Let $\Phi\in Act(G,X)$ be commutative and $m$ be a positive integer. The $m$-times self composition of $\Phi$ is the action $\Phi^{m} :G\times X\rightarrow X$ given by $\Phi^{m}_{g}(x)=\Phi_{g}\Phi_{g}...\Phi_{g}(x)$. The inverse of $\Phi$ is the action $\Phi^{-1}:G\times X\rightarrow X$ given by $(\Phi^{-1})_{g}(x)=\Phi_{g^{-1}}(x)$. 
\label{3.7} 
\end{Definition}  

\begin{theorem}
Let $\Phi\in Act(G,X)$. Then,  
\begin{enumerate}
\item [(i)] $\Phi$ is expansive if and only if $\Phi^{-1}$ is expansive.
\item [(ii)] If $\Phi^{m}$ is expansive for some $m\in\mathbb{Z}\setminus\lbrace 0\rbrace$, then $\Phi$ is expansive. 
\item [(iii)] If $\Phi$ is commutative expansive action, then $\Phi^m$ is expansive for all $m\in \mathbb{Z}\setminus\lbrace 0\rbrace$.  
\end{enumerate}
\label{3.9} 
\end{theorem}
\begin{proof} (i) and (ii) are easy to verify using the definition and the fact that $\Phi^m_g(x)=\Phi_{g^m}(x)$ for all $g\in G$ and all $x\in X$. 
\medskip

(iii) Fix $m$, $i\in \mathbb{N} $ and $h\in G_{i}$. Since $\Phi_{h^{-1}}$ is uniformly continuous, there exists $\gamma^{h}_{i} >0$  such that $d(x,y)\geq \mathfrak{c}$ implies that $d(\Phi_{h}(x),\Phi_{h}(y)) = d(\Phi^{-1}_{h^{-1}}(x),\Phi^{-1}_{h^{-1}}(y)) > \gamma^{h}_{i}$. Since the cardinality of $G_{i}$ is finite, choose $\gamma_{i}=min_{h\in G_{i}}\lbrace \gamma^{h}_{i}\rbrace$ and set $\gamma= min_{i\in \lbrace 1, 2,..., nm\rbrace}\lbrace \gamma_{i}\rbrace$.
Thus $d(x,y)\geq \mathfrak{c}$ implies that $d(\Phi_{h}(x),\Phi_{h}(y))>\gamma$, for all $h\in G_{nm}$.
Let $x\neq y$ in $X$ and $g\in G$ such that $d(\Phi_{g}(x),\Phi_{g}(y)) > \mathfrak{c} $. 
Since $\Phi$ is commutative, we can write $\Phi_{g}= \Phi_{s^{ k_{1}m }_{1}}...\Phi_{s^{ k_{n}m }_{n}}\Phi_{s^{r_{1}}_{1}}...\Phi_{s^{r_{n}}_{n}}$, for some $k_{1},..., k_{n}\in \mathbb{N} \cup \lbrace 0 \rbrace$ and $r_{1},..., r_{n} \in \lbrace 0,...,m-1\rbrace$. As $s^{m-r_{1}}_{1}...s^{m-r_{n}}_{n}\in G_{nm}$, we have  
$d(\Phi_{s^{m-r_{1}}_{1}...s^{m-r_{n}}_{n}}\Phi_{g}(x),\Phi_{s^{m-r_{1}}_{1}...s^{m-r_{n}}_{n}}\Phi_{g}(y))>\gamma$ or $d(\Phi_{s^{(k_{1}+1)m}_{1}}$...$\Phi_{s^{(k_{n}+1)m}_{n}}(x),\Phi_{s^{(k_{1}+1)m}_{1}}$...
$\Phi_{s^{(k_{n}+1)m}_{n}}(y)) > \gamma$ or $d(\Phi^{m}_{s^{k_{1}+1}_{1}...s^{k_{n}+1}_{n}}(x),\Phi^{m}_{{s^{k_{1}+1}_{1}}...s^{k_{n}+1}_{n}}(y))
\\
>\gamma$. Hence, $\Phi^{m}$ is expansive with expansive constant $\gamma$.   Now, use part (i) to complete the proof.
\end{proof}
Recall that a point is said to be periodic under an action if the orbit of that point is finite. Since the set of all periodic points of any expansive homeomorphism on a compact metric space is at most countable, it is obvious to ask whether the set of all periodic points of any expansive action is countable or not. We give an affirmative answer for commutative action of infinite group.   

\begin{theorem}
The set of all periodic points of any commutative expansive action on a compact metric space is at most countable.  
\label{3.11}
\end{theorem}
\begin{proof}
For $\Phi\in Act(G,X)$, set $F(\Phi) = \lbrace x\in X : \Phi_{g}(x) = x$ for all $g\in G\rbrace$. Then, for $m\in \lbrace j > 1 : O(G)$ does not divide $j\rbrace$, set $P(\Phi^m)=\lbrace F(\Phi^{m})\setminus \lbrace F(\Phi^{k}):1\leq k \leq (m-1)$ and $O(G)$ does not divide $k \rbrace$. Since $F(\Phi^{m})$ is a finite set for all $m > 0$ and $P(\Phi)=P(\Phi^m)\cup F(\Phi) \subset \cup_{m\in (\mathbb{N}\setminus \lbrace m : O(G)| m\rbrace} F(\Phi^{m})$, $P(\Phi)$ is atmost countable. To complete the proof, it is enough to show that $x$ is a periodic point if and only if $x\in P(\Phi)$. To see, assume that $x\in P(\Phi^{m})$ for some $m > 0$. Then, $\Phi_{g^{m}}(x) = x$ for all $g\in G$ and hence $\Phi_{s^{m}}(x) = x$ for all $s\in S$. Denote the set $\mathbb{O}(x)=\lbrace \Phi_{s^{j_{1}}_{1}} \Phi_{s^{j_{2}}_{2}}$ ... $\Phi_{s^{j_{n}}_{n}}(x) : j_{i} \in \lbrace 0, 1, m-1 \rbrace$ for $1\leq i\leq n\rbrace$. For $g\in G$, $\Phi_{g}(x) = \Phi_{s^{k_{1}}_{1}} \Phi_{s^{k_{2}}_{2}}$ ... $\Phi_{s^{k_{n}}_{n}}(x) = \Phi_{s^{r_{1}}_{1}} \Phi_{s^{r_{2}}_{2}} ... \Phi_{s^{r_{n}}_{n}}(x)$, where $k_{i} = t_{i}m + r_{i}$ for some $t_{i}\geq 0$, $r_{i}\in \lbrace 0, 1, ... ,m-1 \rbrace$ and $1\leq i\leq n$. Thus, $\Phi_{g}(x)\in \mathbb{O}(x)$ which is a finite set. Hence the orbit of $x$, $\mathcal{O}(x)$ is finite and coincides with $\mathbb{O}(x)$. 
\medskip

Conversely, assume that $x\in X$ is a periodic point, thus the orbit of $x$ is a finite set. Since the order of the group is infinite, the orbit of $x$ is finite and $\Phi_{s}$ is a uniform equivalence, there exists $m\in\mathbb{N}^{+}$ such that $\Phi_{s^m}(x)=x$ for all $s\in S$. Now, for any $g\in G$, $x  = \Phi_{s^{m k_{1}}_{1}} \Phi_{s^{m k_{2}}_{2}}$ ... $\Phi_{s^{m k_{n}}_{n}}(x) = \Phi_{g^{m}}(x)$ for some $k_{i}\geq 0$ and for all $1\leq i\leq n$. Thus, $x\in P(\Phi^{m})\subset P(\Phi)$. 
\end{proof}  

\begin{theorem}
Let $\pi:Y\rightarrow X$ be a locally isometric covering map. Suppose that $\Phi\in Act(G,X)$, $\Psi\in Act(G,Y)$ satisfy $\pi\Psi = \Phi\pi$ i.e. $\pi\Psi_{g}(y) = \Phi_{g}\pi(y)$, for all $y\in Y$ and for all $g\in G$. Suppose that there exists $\delta_{0} >0$ such that for each $y\in Y$ and $0<\delta <\delta_{0}$, $\pi : U_{\delta}(y)\rightarrow U_{\delta}(\pi(y))$ is an isometry. Then 
\begin{enumerate}
\item[(i)] $\Phi$ is expansive if and only if $\Psi$ is expansive.
\item[(ii)] $\Phi$ has the shadowing property if and only if $\Psi$ has the shadowing property.
\end{enumerate}
\label{3.13} 
\end{theorem}
\begin{proof}
(i) Suppose $\Phi$ is expansive with expansive constant $\mathfrak{c}$. We claim that, $0< \gamma<$ min$\lbrace \mathfrak{c}, \delta_{0}\rbrace$ is an expansive constant for $\Psi$. Let $y_{1}, y_{2}\in Y$ such that $p(\Psi_{g}(y_{1}),\Psi_{g}(y_{2}))< \gamma$, for all $g\in G$. Therefore, $\Psi_{g}(y_{2})\in U_{\gamma}(\Psi_{g}(y_{1}))\subset U_{\delta_{0}}(\Psi_{g}(y_{1}))$, for all $g\in G$ which implies that $d(\pi\Psi_{g}(y_{1}), \pi\Psi_{g}(y_{2}))< \gamma$, for all $g\in G$. Hence, $d(\Phi_{g}\pi(y_{1}), \Phi_{g}\pi(y_{2}))< \gamma$, for all $g\in G$. Since $\gamma$ is expansive constant for $\Phi$, we get that $\pi(y_{1}) = \pi(y_{2})$. In particular, for $g=e$ the identity element of $G$, $y_{2}\in U_{\gamma}(y_{1})\subset U_{\delta_{0}}(y_{1})$ and since $\pi$ is an isometry on $U_{\gamma}(y_{1})$, we get that $0= d(\pi(y_{1}),\pi(y_{2})) = p(y_{1}, y_{2})$ or $y_{1} = y_{2}$. Hence, $\Psi$ is expansive with expansive constant $\gamma$. 
\medskip

Conversely, suppose that $\Psi$ is expansive with expansive constant $\mathfrak{c}$. Since $\Psi$ is uniformly continuous, we can choose $0< \delta< \delta_{0}$ such that $p(y_{1},y_{2}) < \delta$ implies that $p(\Psi_{s}(y_{1}),\Psi_{s}(y_{2})) < \delta_{0} $, for all $s\in S$. Let $0< \gamma <$ min$\lbrace \mathfrak{c}, \delta\rbrace$ and $x_{1}, x_{2}\in X$ such that $d(\Phi_{g}(x_{1}), \Phi_{g}(x_{2}))< \gamma$, for all $g\in G$. Choose, $y_{1}, y_{2}\in Y$ with $\pi(y_{1})=x_{1}$, $\pi(y_{2})=x_{2}$ and $p(y_{1}, y_{2})= d(x_{1}, x_{2})< \gamma <\delta <\delta_{0}$. 
If $d(\Phi_{g}(x_{1}), \Phi_{g}(x_{2}))\leq \gamma$ for all $g\in G$, then $p(\Psi_{s}(y_{1}), \Psi_{s}(y_{2})) = d(\pi\Psi_{s}(y_{1}), \pi\Psi_{s}(y_{2})) = d(\Phi_{s}(x_{1}), \Phi_{s}(x_{2})) \leq \gamma$ for all $s\in G_{1}$. 
By induction on $n$, we can show that $p(\Psi_{g}(y_{1}), \Psi_{g}(y_{2})) \leq \gamma$ for all $g\in G_{n}$. Thus $p(\Psi_{g}(y_{1}), \Psi_{g}(y_{2})) \leq \gamma$ for all $g\in G$ and hence by expansivity of $\Psi$, we get that $y_{1} = y_{2}$ or $x_{1} = x_{2}$. Thus, $\Phi$ is expansive with expansivity constant $\gamma$
\medskip

(ii) Suppose that $\Phi$ has the shadowing property. For any $0< \epsilon\leq \delta_{0}$, the shadowing property of $\Phi$ gives $0< \delta <\epsilon$ such that every $\delta$-pseudo orbit can be $\epsilon$-traced. Let $f:G\rightarrow Y$ be a $\delta$-pseudo orbit for $\Psi$ i.e. $p(\Psi_{s}(f(g)), f(sg)) < \delta <\delta_{0}$, for all $s\in S$ and all $g\in G$. Therefore, $f(sg)\in U_{\delta_{0}}(\Psi_{s}f(g))$. Since $\pi$ is an isometry on $U_{\delta_{0}}(\Psi_{s}f(g))$, we get that
$d(\pi \Psi_{s}(f(g)), \pi f(sg)) = p(\Psi_{s}(f(g)), f(sg)) < \delta <\delta_{0}$ or $d(\Phi_{s}\pi f(g), \pi f(sg)) = p(\Psi_{s}(f(g)), f(sg)) < \delta <\delta_{0}$, for all $s\in S$ and all $g\in G$. So, $h:G\rightarrow X$ given by $h(g)= \pi f(g)$ is a $\delta$-pseudo orbit for $\Phi$ and hence, there exists $x\in X$ such that $d(\Phi_{g}(x),\pi f(g))< \epsilon \leq \delta_{0}$ or $d(\Phi_{g} \pi(y), \pi f(g))< \epsilon \leq \delta_{0}$ or $d(\pi \Psi_{g}(y), \pi f(g))< \epsilon \leq \delta_{0}$ for some $y\in Y$ and for all $g\in G$ which implies $p(\Psi_{g}(y), f({g}))< \epsilon$ for all $g\in G$. Hence $f(g)$ is $\epsilon$-traced by $y\in Y$. 
\medskip

Conversely, suppose that $\Psi$ has the shadowing property. For $0< \epsilon\leq \delta_{0}$, the shadowing property of $\Psi$ gives $0< \delta <\epsilon$ such that every $\delta$-pseudo orbit of $\Psi$ can be $\epsilon$-traced. Let $f:G\rightarrow X$ be a $\delta$-pseudo orbit for $\Phi$ i.e. $d(\Phi_{s}f(g), f(sg)) < \delta <\delta_{0}$, for all $s\in S$ and all $g\in G$. Since $\pi$ is onto map, there exists $t:G\rightarrow Y$ such that $\pi t(g) = f(g)$ and $d(\Phi_{s}\pi (t(g)), \pi t(sg))< \delta$ or $d(\pi \Psi_{s}(t(g)), \pi t(sg))< \delta$. Since $\pi$ is an isometry on $U_{\delta_{0}}(\Psi_{s} (t(g)))$, we get that $d(\pi \Psi_{s}(t(g)), \pi t(sg)) = p(\Psi_{s}(t(g)), t(sg)) < \delta <\delta_{0}$. So, $t$ is a $\delta$-pseudo orbit of $\Psi$. Then, there exists $y\in Y$ such that $p(\Psi_{g}(y), t(g))< \epsilon \leq \delta_{0}$ which implies that
$d(\pi \Psi_{g}(y), \pi t(g))< \epsilon \leq \delta_{0}$, for some $y\in Y$ or $d(\Phi_{g}\pi (y), \pi t(g))< \epsilon \leq \delta_{0}$ or $d(\Phi_{g}(x), f(g))< \epsilon$, for all $g\in G$. Hence $f$ is $\epsilon$-traced by $x\in X$.  
\end{proof}

\section{Gromov-Hausdorff Stability} 

In \cite{A3}, authors defined the Gromov-Hausdorff distance between $(X,d_X)$ and $(Y,d_Y)$ by $d_{GH}(X, Y) =$ inf$\lbrace \delta > 0 \mid $ there exists $\delta$-isometries $i: X\rightarrow Y$ and $j: Y\rightarrow X\rbrace$.  
\medskip

For a finite symmetric generating set $S$ of $G$ and $\Phi,\Psi\in Act(G, X)$, the map $d_{S}^{X}(\Phi, \Psi) =$ sup$_{(s, x)\in S\times X}d^{X}(\Phi_{s}(x), \Psi_{s}(x))$ is a metric. 
\medskip

For non-necessarily continuous map $i:X\rightarrow Y$, $\Phi\in Act(G, X)$ and $\Psi\in Act(G, Y)$, we define 
$d_{S}^{Y}(\Psi\circ i,i\circ \Phi) =$ sup$_{(s, x)\in S\times X}d^{Y}(\Psi_{s}(i(x)), i(\Phi_{s}(x)))$.  
\medskip

The $G$-Gromov-Hausdorff distance between $\Phi\in Act(G,X)$ and $\Psi\in Act(G,Y)$ with respect to the generating set $S$ is defined as 
\begin{center}
$d_{GH}^{S}(\Phi, \Psi) =$ inf$\lbrace \delta > 0\mid$ there exists $\delta-$isometries $i: X\rightarrow Y$ and $j: Y\rightarrow X$ such that $d_{S}^Y(\Psi\circ i, i\circ \Phi) < \delta$ and $d_{S}^X(\Phi\circ j, j\circ \Psi) < \delta \rbrace$.   
\end{center} 
\medskip

The strong $G$-Gromov-Hausdorff distance between $\Phi\in Act(G, X)$ and $\Psi\in Act(G, Y)$ with respect to the generating set $S$ is defined as  
\begin{center}
$\overline{d}_{GH}^{S}(\Phi, \Psi) =$ inf$\lbrace \delta > 0\mid$ there exists $\delta-$isometry $i: Y\rightarrow X$ such that $d_{S}^X(\Phi\circ i, i\circ \Psi)<\delta\rbrace$.  
\end{center} 
\medskip

A function $d : X\times X\rightarrow \mathbb{R}$  is said to be a pseudo-quasimetric on $X$, if for every triplet $(x, y, z)$ of $X\times X\times X$ following conditions hold 
\begin{enumerate}
\item [(i)] $d(x, y) \geq 0$ and $d(x, x) = 0$;
\item [(ii)] $d(x, y) = d(y, x)$;
\item [(iii)] There exists a constant $k\geq 1$ such that $d(x, y)\leq k[d(x, z) + d(z,y)]$, where $k$ is called the coefficient of pseudo-quasi metric. 
\end{enumerate}

If only (i) and (iii) hold, then $d$ is called an asymmetric pseudo-quasi metric on $X$.  

\begin{theorem}
Let $S$ be a generating set of $G$. For $\Phi\in Act(G, X)$ and $\Psi\in Act(G, Y)$ following statements are true
\begin{enumerate} 
\item [(i)] $d_{GH}^{S}(\Phi, \Psi) = $max $\lbrace \overline{d}_{GH}^{S}(\Phi, \Psi), \overline{d}_{GH}^{S}(\Psi, \Phi)\rbrace$ and if $X = Y$, then $d_{GH}^{S}(\Phi, \Psi) \leq d_{S}(\Phi, \Psi)$;
\item [(ii)] $d_{GH}(X, Y) \leq d_{GH}^{S}(\Phi, \Psi)$ and $d_{GH}(X, Y) = d_{GH}^{S}(Id_{X}, Id_{Y})$, where $Id_{X}, Id_{Y}$ denote the trivial actions on $X$ and $Y$ respectively; 
\item [(iii)] $d_{GH}^{S}(\Phi, \Psi) \geq 0$ and $d_{GH}^{S}(\Phi, \Phi) = 0$;
\item [(iv)] $d_{GH}^{S}(\Phi, \Psi) = d_{GH}^{S}(\Psi, \Phi)$;
\item [(v)] If $\Lambda \in Act(G, Z)$, then $d_{GH}^{S}(\Phi, \Psi)\leq 2[d_{GH}^{S}(\Phi, \Lambda) + d_{GH}^{S}(\Lambda,\Psi)]$.
\item [(vi)] $\overline{d}_{GH}^{S}(\Phi, \Psi) \geq 0$ and $\overline{d}_{GH}^{S}(\Phi, \Psi) = 0$;
\item [(vii)] If $\Lambda \in Act(G, Z)$, then $\overline{d}_{GH}^{S}(\Phi, \Psi)\leq 2[\overline{d}_{GH}^{S}(\Phi, \Lambda) + \overline{d}_{GH}^{S}(\Lambda,\Psi)]$. 
\label{6.1}
\end{enumerate}
\end{theorem}

\begin{proof}
Proof of (i), (ii), (iii), (iv) and (vi) follow from the corresponding definitions. The proof of (vii) is similar to the proof of (v). So, we prove only (v). For fixed $\epsilon > 0$, we can choose $\delta_{1}$-isometries $i : X\rightarrow Z$, $j: Z\rightarrow X$ and $\delta_{2}$-isometries $k: Y\rightarrow Z$, $l : Z\rightarrow Y$ such that $\delta_{1} < d_{GH}^{S}(\Phi, \Lambda) + \epsilon$, $\delta_{2} < d_{GH}^{S}(\Lambda, \Psi) + \epsilon$, $d_{S}^Z(\Lambda\circ i, i\circ \Phi) < \delta_{1}$, $d_{S}^X(j\circ \Lambda, \Phi\circ j) < \delta_{1}$, $d_{S}^Z(\Lambda\circ k, k\circ \Psi) < \delta_{2}$, $d_{S}^Y(l\circ \Lambda, \Psi\circ l) < \delta_{2}$. Clearly, $l\circ i: X\rightarrow Y$ and $j\circ k: Y\rightarrow X$ are $(\delta_{1}+\delta_{2})$-isometries and hence, $2(\delta_{1}+\delta_{2})$-isometries as well. For all $x\in X$, $s\in S$, we have 
\begin{align*} 
&d^{Y}(\Psi_{s}(l(i(x))), l(i(\Phi_{s}(x))))\\
&\leq d^{Y}(\Psi_{s}(l(i(x))), l(\Lambda_{s}(i(x)))) + d^{Y}(l(\Lambda_{s}(i(x))), l(i(\Phi_{s}(x)))) \\
&< \delta_{2} + d^{Z}(\Lambda_{s}(i(x)), i(\Phi_{s}(x))) + \delta_{2}\\ 
&< \delta_{1} + 2\delta_{2}< 2(\delta_{1} + \delta_{2}).  
\end{align*} 

Hence, $d_{S}^Y(\Psi\circ (l\circ i), (l\circ i)\circ \Phi) < 2(\delta_{1} + \delta_{2})$. Similarly, one can prove that $d_{S}^X((j\circ k) \circ \Psi, \Phi\circ (j\circ k)) < 2(\delta_{1}+\delta_{2})$. Therefore, $d_{GH}^{S}(\Phi, \Psi) \leq 2(\delta_{1}+\delta_{2}) < 2(d_{GH}^{S}(\Phi, \Lambda) + d_{GH}^{S}(\Lambda, \Psi) + 2\epsilon)$. Since $\epsilon$ was chosen arbitrarily, we get the result. 
\end{proof}

Items (iii), (iv) and (v) imply that $d_{GH}^{S}$ forms a pseudo-quasimetric and items (vi), (vii) imply that the $\overline{d}_{GH}^{S}$ forms an asymmetric pseudo-quasimetric. 
\medskip 

In this section, without loss of generality we assume that $0 < \epsilon, \delta, \eta, \gamma, \mathfrak{c} < \frac{1}{2}$. 

\begin{Definition}
An action $\Phi\in Act(G, X)$ is said to be strongly $GH$-stable (resp. $GH$-stable) with respect to the generating set $S$ of $G$ if for every $\epsilon > 0$ there exists $\delta > 0$ such that for every action $\Psi\in Act(G, Y)$ satisfying $\overline{d}_{GH}^{S}(\Phi, \Psi) < \delta$ (resp. $d_{GH}^{S}(\Phi, \Psi) < \delta$) there exists a continuous $\epsilon$-isometry $k: Y\rightarrow X$ such that $\Phi\circ k = k\circ \Psi$.   
\label{6.2} 
\end{Definition}

The following result shows that strong $GH$-stability is independent of the generating set. 

\begin{theorem}
Let $S$ and $P$ be two finite symmetric generating sets of $G$. If $\Phi\in Act(G, X)$, then $\Phi$ is strongly $GH$-stable with respect to $S$ if and only if $\Phi$ is strongly $GH$-stable with respect to $P$.  
\label{6.3}
\end{theorem}
\begin{proof}
Suppose that $\Phi$ is strongly $GH$-stable with respect to the generating set $S$. Then, for every $\epsilon > 0$ there exists $\eta> 0$ such that for each $\Psi\in Act(G, Y)$ satisfying $\overline{d}_{GH}^{S}(\Phi, \Psi) < \eta$ there exists a continuous $\epsilon$-isometry $k: Y\rightarrow X$ such that $\Phi\circ k = k\circ \Psi$. Set $m = $max$_{s\in S}d_{P}(s)$, where $d_{P}$ is the word length metric on $G$ induced by $P$. Choose $\gamma > 0$ such that $m\gamma < \eta$. Fix $\delta > 0$ such that $d^{X}(x_{1}, x_{2}) < \delta$ implies $d^{X}(\Phi_{g}(x_{1}), \Phi_{g}(x_{2})) < \gamma$ for all $g\in G_{m}^{P}$, where $G_{m}^{P}$ is the set of all elements of length at most $m$ with respect to the generating set $P$.  
\medskip

For any $s\in S$, we can write $s = p_{1}p_{2}...p_{d_{P}(s)}$, where $d_{P}(s) \leq m$ and $p_{i}\in P$ for all $1\leq i\leq d_{P}(s)$. Let $\Psi\in Act(G, Y)$ with $\overline{d}_{GH}^{P}(\Phi, \Psi) < \delta$. Thus, there exists $\delta$-isometry $i:Y\rightarrow X$ such that $d_{P}^X(\Phi\circ i, i\circ \Psi)< \delta$. Since every $\delta$-isometry is $\eta$-isometry, it is enough to show that $d_{S}^X(\Phi\circ i, i\circ \Psi)< \eta$. Note that for all $y\in Y$, $s\in S$ following holds.
\medskip
 
$
d^{X}(\Phi_{s}\circ j(y), j\circ \Psi_{s}(y))\\
= d^{X}(\Phi_{p_{1}...p_{d_{P}(s)}}\circ j(y), j\circ \Psi_{p_{1}...p_{d_{P}(s)}}(y)) \\
\leq d^{X}(\Phi_{p_{1}...p_{d_{P}(s)-1}}(\Phi_{p_{d_{P}(s)}}\circ j(y)), \Phi_{p_{1}...p_{d_{P}(s)-1}}(j\circ \Psi_{p_{d_{P}(s)}}(y))) \\
+ d^{X}(\Phi_{p_{1}...p_{d_{P}(s)-2}}(\Phi_{p_{d_{P}(s)-1}}\circ j (\Psi_{p_{d_{P}(s)}}(y))),\Phi_{p_{1}...p_{d_{P}(s)-2}}(j\circ \Psi_{p_{d_{P}(s)-1}} (\Psi_{p_{d_{P}(s)}}(y))) \\
+ ... + d^{X}(\Phi_{p_{1}}\circ j (\Psi_{p_{2}...p_{d_{P}(s)}}(y)), j\circ \Psi_{p_{1}...p_{d_{P}(s)}}(y))< m\gamma< \eta$
\end{proof}

\begin{theorem}
Let $\Phi\in Act(G, X)$ be strongly $GH$-stable. If $\Psi\in Act(G, Z)$ is isometric to $\Phi$ via isometry $k : Z \rightarrow X$, then $\Psi$ is strongly $GH$-stable.  
\label{6.5}
\end{theorem}

\begin{proof}
Let $\epsilon > 0$ be fixed and let $\delta > 0$ be given for $\epsilon$ by strong $GH$-stability of $\Phi$. Let $\Lambda\in Act(G, Y)$ such that $\overline{d}_{GH}^{S}(\Psi, \Lambda) < \frac{\delta}{2}$. By Theorem \ref{6.1}, $\overline{d}_{GH}^{S}(\Phi, \Lambda)\leq 2[\overline{d}_{GH}^{S}(\Phi, \Psi)$ $+ \overline{d}_{GH}^{S}(\Psi, \Lambda)] < 2[0 +\frac{\delta}{2}] = \delta$. Since $\Lambda\in Act(G, Y)$, there exists a continuous $\epsilon$-isometry $h: Y\rightarrow X$ such that $\Phi\circ h = h\circ \Lambda$ and hence, $\Psi\circ k^{-1}\circ h = k^{-1}\circ h\circ \Lambda$. Since $k$ is an isometry, one can see that $k^{-1}h: Y\rightarrow X$ is a continuous $\epsilon$-isometry. Hence, $\Psi$ is strongly $GH$-stable.  
\end{proof}

\begin{Lemma}
Let $X$ be a metric space in which every closed ball is compact and $\Phi$ be expansive with expnaisve constant $\mathfrak{c}$. Then, for each $x\in X$ and every $\epsilon>0$ there is $n_{\epsilon}(x)\geq 0$ such that sup$_{g\in G_{n_{\epsilon}(x)}}d^{X}(\Phi_{g}(x), \Phi_{g}(x'))$ $\leq \mathfrak{c}$ implies that $d^{X}(x, x')< \epsilon$.
\label{4.5} 
\end{Lemma} 
\begin{proof}
If not then there exists $\epsilon > 0$ such that for each $n\geq 0$ there exists $x'_{n}\in X$ satisfying sup$_{g\in G_{n}}d^{X}(\Phi_{g}(x), \Phi_{g}(x'_{n}))\leq \mathfrak{c}$ and $d^{X}(x, x'_{n}) \geq \epsilon$. Since closed ball of radius $\mathfrak{c}$ around $x$ is compact, without loss of generality we may assume that $x'_{n}\rightarrow x'$. Since $G = \cup_{n\geq 0}G_{n}$, we get that $d^{X}(\Phi_{g}(x), \Phi_{g}(x'))\leq \mathfrak{c}$ for all $g\in G$ and $d^{X}(x, x')\geq \epsilon$, which contradicts the expansivity of $\Phi$. 
\end{proof}

\begin{theorem}
Let $X$ be a metric space in which every closed ball is compact. If $\Phi\in Act(G, X)$ is expansive and has the shadowing property, then it is strongly $GH$-stable. 
\end{theorem}

\begin{proof}
Let $\mathfrak{c}$ be an expansivity constant for $\Phi\in Act(G,X)$ and let us fix generating set $S$ of $G$. Let $\epsilon>0$ be given and we choose $0 < \eta < \frac{1}{8}$min$\lbrace \epsilon, \mathfrak{c} \rbrace$. Let $0 < \delta < \eta$ be given for $\eta$ by the shadowing property of $\Phi$. Let $Y$ be a metric space and let $\Psi\in Act(G, Y)$ be such that $\overline{d}_{GH}^{S}(\Phi, \Psi) < \delta$. Thus, we can choose $\delta$-isometry  $i: Y\rightarrow X$ satisfying $d_{S}^{X}(\Phi\circ i, i\circ \Psi) < \delta$. 
\medskip

For each $y\in Y$, set $x_{g}^{y} = i(\Psi_{g}(y))$ for all $g\in G$. So, for all $g\in G$, $s\in S$ and each $y\in Y$, 
$d^{X}(x_{sg}^{y}, \Phi_{s}(x_{g}^{y}))= d^{X}(i(\Psi_{s}\Psi_{g}(y)), \Phi_{s}(i(\Psi_{g}(y))))= d^{X}(i\circ \Psi_{s}(\Psi_{g}(y)), \Phi_{s}\circ i(\Psi_{g}(y))) < \delta.$ 
\medskip

Therefore by expansivity and the shadowing property, there exists a unique $\eta$-tracing point $x^{y}$ of $\lbrace x_{g}^{y}\rbrace$ for each $y\in Y$. In particular, $d^{X}(\Phi_{g}(x^{y}), x_{g}^{y}) < \frac{\mathfrak{c}}{2}$ for all $g\in G$. Define $h : Y\rightarrow X$ by $h(y) = x^{y}$ satisfying $d^{X}(\Phi_{g}(h(y)), i(\Psi_{g}(y))) \leq \eta$ for all $y\in Y$, $g\in G$. 
\medskip

For $g = e$, we have $d^{X}(h(y), i(y))\leq \eta$ for all $y\in Y$. Then, for all $y_{1}, y_{2}\in Y$, we have
\begin{align*}
&|d^{X}(h(y_{1}), h(y_{2})) - d^{Y}(y_{1}, y_{2})| \\
&\leq  |d^{X}(h(y_{1}), h(y_{2})) - d^{X}(i(y_{1}), i(y_{2}))| + |d^{X}(i(y_{1}), i(y_{2})) - d^{Y}(y_{1}, y_{2})| \\
&\leq  |d^{X}(h(y_{1}), h(y_{2})) - d^{X}(h(y_{1}), i(y_{2}))|+ |d^{X}(h(y_{1}), i(y_{2})) - d^{X}(i(y_{1}), i(y_{2}))| + \delta \\ 
&\leq d^{X}(h(y_{2}), i(y_{2})) + d^{X}(h(y_{1}), i(y_{1})) + \delta \\
&\leq 2\eta +\delta < \epsilon
\end{align*}

Further, $d_{H}^{X}(h(Y), X)\leq d_{H}^{X}(h(Y), i(Y)) + d_{H}^{X}(i(Y), X))\leq \eta + \delta < \epsilon$. Thus, it follows that $h$ is an $\epsilon$-isometry. 
\medskip

For all $y\in Y$, $d^{X}(\Phi_{g}(h(\Psi_{s}(y))), i(\Psi_{g}(\Psi_{s}(y)))) \leq \eta$, $d^{X}(\Phi_{g}(\Phi_{s}(h(y))), i(\Psi_{gs}(y))) \leq \eta$. Therefore, by expansivity of $\Phi$ we get that $\Phi_{s}\circ h = h\circ \Psi_{s}$ for all $s\in S$. In other words, we have $\Phi\circ h = h\circ \Psi$.
\medskip

Let $y\in Y$ and $h(y) = x$. Further, let $\gamma > 0$ and $n_{\gamma}(x)$ is the number in Lemma \ref{4.5} corresponding to $\gamma$. Let $\delta > 0$ be such that $d^Y(x,y)<\delta$ implies that $d^Y(\Psi_g(x),\Psi_g(y))<\frac{\eta}{8}$ for all $g\in n_{\gamma}(x)$. Then, for all $g\in G_{n_{\gamma}(x)}$ and $y'\in Y$ such that $d^{Y}(y, y') < \delta$, we get that 
\medskip

$d^{X}(\Phi_{g}(h(y)), \Phi_{g}(h(y')))
\\
=  d^{X}(h(\Psi_{g}(y)), h(\Psi_{g}(y'))) 
\\
\leq d^{X}(h(\Psi_{g}(y)), i(\Psi_{g}(y'))) + d^{X}(i(\Psi_{g}(y)), i(\Psi_{g}(y')))+ d^{X}(i(\Psi_{g}(y')), h(\Psi_{g}(y')))
\\
\leq 2\eta + \delta + d^{Y}(\Psi_{g}(y),\Psi_{g}(y'))
\\
\leq 3\eta + \frac{\eta}{8}
\\
< \mathfrak{c}$
\medskip

This implies $d^{X}(h(y), h(y')) < \gamma$ and hence, $h$ is a continuous map. Thus, we conclude that $\Phi$ is strongly $GH$-stable. 
\end{proof}

\begin{Corollary}
Every expansive action with the shadowing property on a relatively compact metric space is strongly $GH$-stable and hence, $GH$-stable. 
\end{Corollary}

\begin{Example}
Let $G = \langle a, b : ba =a^{2}b \rangle$ be a solvable group with the symmetric generating set $S = \lbrace a, a^{-1}, b, b^{-1}\rbrace$ and let $F$ be an expansive homeomorphism with the shadowing property. Let $\mathfrak{c}$ be an expansivity constant for $F$. Further, assume that $F^{-1}$ is a contraction map i.e. there exists $0\leq K < 1$ such that $d(F^{-1}(x), F^{-1}(y)) \leq K d(x, y)$ for all $x, y\in X$. Let $\Phi\in Act(G,X)$ be generated by $\Phi_{a}(x) = x$ and $\Phi_{b}(x) = F(x)$ for all $x\in X$. It is then clear that $\Phi$ is an expansive action. Moreover, we claim that $\Phi$ has the shadowing property.  
\medskip

Let $\epsilon>0$ be given and $2\gamma = $min$\lbrace \mathfrak{c}, \epsilon\rbrace$. Further, let $\delta > 0$ be such that every $\delta$-pseudo orbit of $F$ is $\gamma$-traced. Let $f: G\rightarrow X$ be a $\delta$-pseudo orbit of $\Phi$. Now, for each $g\in G$, $\lbrace x_{i} = f(b^{i}g)\rbrace_{i\in \mathbb{Z}}$ is a $\delta$-pseudo orbit of $F$. Let $x_{g}$ be its $\gamma$-tracing point i.e. $d(F^{i}(x_{g}),f(b^{i}g))<\gamma$. Since $d(F^{i}(x_{bg}), f(b^{i}bg)) < \gamma$, $d(F^{i}(F(x_{g})), f(b^{i+1}g)) < \gamma$, we have that $d(F^{i}(x_{bg}), F^{i}(F(x_{g})) < 2\gamma \leq \mathfrak{c}$ for all $i\in \mathbb{Z}$. By expansivity of $F$, we have $x_{bg} = F(x_{g}) = \Phi_{b}(x_{g})$ for all $g\in G$. Since $\Phi_{a}(x) = x$, $d(f(g), f(ag)) < \delta$ for all $g\in G$. Then, $b^{i}a = a^{2^{i}}b^{i}$ for $i \geq 1$ implies that $d(f(a^{2i-k-1}b^{i}g), f(a^{2i-k}b^{i}g)) < \delta$ for $0\leq k\leq 2i$. Hence, $d(f(a^{2i}b^{i}g), f(b^{i}g)) < \delta(2i)$ for all $g\in G$. Since $f(a^{2^{i}}b^{i}g) = f(b^{i}ag)$, we have $d(F^{i}(x_{ag}), f(b^{i}ag)) < \gamma$ and $d(F^{i}(x_{g}), f(b^{i}g)) < \gamma$ for all $g\in G$, $i\geq 0$. Now since, $F^{-1}$ is a contraction map, we have that $d(x_{ag}, x_{g}) \leq K^{i}d(F^{i}(x_{ag}), F^{i}(x_{g})) \leq K^{i}(2\gamma + \delta(2i))$ for all $i\geq 0$. As $i\rightarrow \infty$, right hand side of the last inequality tends to zero, which implies that $x_{ag} = x_{g}$ for all $g\in G$. Thus, we have that $x_{bg} = F(x_{g}) = \Phi_{b}(x_{g})$ and $x_{ag} = \Phi_{a}(x_{g})$ for all $g\in G$. Therefore, $d(\Phi_{g}(x_{e}), f(g)) < \epsilon$ for all $g\in G$ and hence, we conclude that $\Phi$ has the shadowing property. 
\label{E4.8}
\end{Example}

\section{Sequential Shadowing Property} 

Smale's spectral decomposition theorem (Theorem 3.1.11 \cite{A1}) is one of the celebrated result in topological dynamics. Existence of such a decomposition is still not proved for finitely generated group actions satisfying expansivity and the shadowing property. In this section, we have introduced another variant of the shadowing property, called sequential shadowing property which surprisingly does not need the support of expansive property to spectrally decompose the system.                  
\medskip

\begin{Definition}
Let $x,y\in X$ and $\delta>0$. We say that a finite sequence $\lbrace x = x_{0}, x_{1},. . ., x_{k} =y\rbrace_{i=0}^{k}$ of elements of $X$ is a weak $\delta$-chain of $\Phi\in Act(G,X)$ from $x$ to $y$ with respect to the generating set $S$  if there is a sequence $\lbrace s_{x_i}\rbrace_{i=0}^{k-1}$ of elements of $S$ such that $d(\Phi_{s_{x_i}}(x_{i}), x_{i+1}) < \delta$ for $0\leq i\leq (k-1)$. We say that $x$ is weakly $\delta$-related to $y$ (written as $x WR^{S}_{\delta} y$) with respect to $S$, if there is weak $\delta$-chains from $x$ to $y$ and from $y$ to $x$. If $x WR^{S}_{\delta} y$ for each $\delta>0$, then $x$ is said to be weakly related to $y$ (written $x WR^{S} y$) with respect to $S$. Action $\Phi\in Act(G,X)$ is said to be weak chain transitive with respect to $S$ if $x WR^{S} y$ for all $x, y\in X$.
\label{5.2} 
\end{Definition}

For $\delta > 0$, $s\in S$ and $x\in X$, the sequence $\lbrace x, \Phi_{s}(x), x\rbrace$ is a weak $\delta$-chain of $\Phi\in Act(G,X)$ from $x$ to itself i.e. $x WR^{S} x$ for all $x\in X$. Thus, the set $WCR^{S}(\Phi)=\lbrace x\in X\mid x WR^{S} x\rbrace = X$ is independent of the generating set $S$. It is easy to check that the relation $WR^{S}$ is an equivalence relation on $WCR^{S}(\Phi) = X$. Thus, we can write $X = \cup_{\lambda\in \Lambda}B^{S}_{\lambda}$, where $B^{S}_{\lambda}[x]=\lbrace y : y WR^{S} x\rbrace$ is the equivalence class containing $x$. 

\begin{Definition}
A sequence $\lbrace x_{i}\rbrace_{i=-\infty}^{\infty}$ of elements of $X$ is said to be $\delta$-sequential pseudo orbit (SPO) of $\Phi\in Act(G,X)$ with respect to $S$ if there is a sequence $\lbrace s_{x_i}\rbrace_{i=-\infty}^{\infty}$ of elements of $S$ such that $d(\Phi_{s_{x_i}}(x_{i}), x_{i+1}) < \delta$ for all $i\in \mathbb{Z}$. A sequence $\lbrace x_{i}\rbrace_{i=-\infty}^{\infty}$ is $\epsilon$-sequentially traced (ST) if there is $x\in X$ and a sequence $\lbrace g_{i} \rbrace _{i = -\infty}^{\infty}$ with $g_{0} = e$ and $g_i\neq e$ for all $i\neq 0$, such that $d(\Phi_{g_{i}}(x), x_{i}) < \epsilon$ for all $i\in \mathbb{Z}$. Action $\Phi$ is said to have sequential shadowing property (SSP) if for every $\epsilon>0$ there is $\delta>0$ such that every $\delta$-SPO is $\epsilon$-ST by some point in $X$.   
\label{5.3}
\end{Definition} 

In case of homeomorphisms SSP suggests us that numerically found trajectories with uniformly small one-step forward or backward errors at each step, can be seen through a part of real trajectory. It helps us to understand the larger class of pseudo-trajectories by studying the original trajectories.

\begin{theorem}
Let $\Phi\in Act(G,X)$. Then,
\\
(i) SSP and weak chain transitivity of $\Phi$ are independent of the generating set.   
\\
(ii) SSP and weak chain transitivity of $\Phi$ are invariant under uniform conjugacy.  
\label{5.4} 
\end{theorem}
  
\begin{proof}  
(i) Let $S$ and $S'$ be two generating sets of $G$. Since any two word norms are bi-Lipschitz equivalent, we can choose a constant $C \geq 1$ such that $\frac{|g|_{S'}}{C}\leq |g|_{S} \leq C|g|_{S'}$ for all $g\in G$. Let $x,y\in X$ be such that $d(\Phi_{s}(x), y) < \delta$. If $s = s'_{1}s'_{2}...s'_{C}$ for $s'_{i}\in S'$, $1 \leq i \leq C$ then, $\lbrace x, \Phi_{s'_{C}}(x), \Phi_{s'_{C-1}s'_{C}}(x),..., \Phi_{s'_{2}s'_{C-1}s'_{C}}(x), y\rbrace$ forms a weak $\delta$-chain with respect to $S'$. So, if $\lbrace x_{0}, x_{1}, ..., x_{k}\rbrace$ is weak $\delta$-chain with respect to $S$, then using the same methodology it can be extended to weak $\delta$-chain with respect to $S'$. Similarly, any $\delta$-SPO with respect to $S$ can be extended to $\delta$-SPO with respect to $S'$. Therefore, we can conclude that SSP and weak chain transitivity of $\Phi$ are independent of the generating set.

(ii) The proof of this part is easy to work out.    
\end{proof} 

\begin{theorem}
Let $X$ be a compact metric space. The trivial action $\Phi\in Act(G, X)$ has SSP if and only if $X$ is totally disconnected.
\label{5.9}
\end{theorem}
\begin{proof}
Suppose that $\Phi$ has SSP and $X$ has at least one non-trivial connected component $C$. Let $x$,$y$ be two distinct points in $C$ and $\epsilon= \frac{d(x, y)}{3}$. Let $\delta$ be given for $\epsilon$ by SSP of $\Phi$. Further, let $z_0 = x,$...$,z_k = y$ be a sequence of points from $x$ to $y$ so that $d(z_{j+1},z_j) < \delta$ for $j = 0,$...$,k-1$. Such a chain exists by the connectedness of $C$. Let $s\in S$ be fixed and
\[
p_{i} =
	\begin{cases}
	\Phi_{s^{i}}(x) & i < 0 \\
	\Phi_{s^{i}}(z_{i}) & 0\leq i\leq k\\
	\Phi_{s^{i}}(y) & i > k
	\end{cases}
\]
Since for $i = 0,$...$,k-1$, $d(\Phi_{s}(p_{i}), p_{i+1}) = d(\Phi_{s}\Phi_{s^{i}}(z_{i}), \Phi_{s^{i+1}}) = d(z_{i}, z_{i+1}) < \delta$, $\lbrace p_{i}\rbrace_{i=-\infty}^{\infty}$ is $\delta$-SPO. Suppose $\lbrace p_i\rbrace$ is $\epsilon$-ST by a point $p\in X$ via the sequence $\lbrace g_{i}\rbrace_{i=-\infty}^{\infty}$. Thus, by definition of $\epsilon$-ST, we get that $d(p, x) < \epsilon$ and $d(\Phi_{g_{k}}(p), \Phi_{s^{p_{k}}}(y)) = d(p, y) < \epsilon$. Then, $0 < \epsilon = \frac{d(x, y)}{3} \leq \frac{d(x, p) + d(p, y)}{3} < \frac{2\epsilon}{3}$, which is absurd. So, $X$ must be totally disconnected. 
\medskip

Conversely, if $X$ is totally disconnected, then for each $\epsilon>0$ there exists a finite open cover $U_{0},..., U_{m}$ of $X$ such that $U_{i}\cap U_{j}=\phi$, for $i\neq j$ and $diam(U_{i})<\epsilon$, for each $i$, $j\in \lbrace 0 , 1 ,... , m\rbrace$. 
Choose $\delta$ such that $0<\delta <$ min$\lbrace d(U_{i},U_{j}):i\neq j)\rbrace$ and choose a $\delta$-pseudo orbit $\lbrace x_{i}\rbrace_{i=-\infty}^{\infty}$ of $I$. Then, there is $U_{j}$ such that $\lbrace x_{i}\rbrace_{i=-\infty}^{\infty} \subset U_{j}$ and hence any point of $U_{i}$ can $\epsilon$-trace this orbit. 
\end{proof}

\begin{theorem}
Let $\pi:Y\rightarrow X$ be a locally isometric covering map. Suppose that $\Phi\in Act(G,X)$, $\Psi\in Act(G,Y)$ satisfy $\pi\Psi = \Phi\pi$ i.e. $\pi\Psi_{g}(y) = \Phi_{g}\pi(y)$, for all $y\in Y$, $g\in G$. Further, suppose that there exists $\delta_{0} >0$ such that for each $y\in Y$ and $0<\delta <\delta_{0}$, $\pi : U_{\delta}(y)\rightarrow U_{\delta}(\pi(y))$ is an isometry. Then, $\Phi$ has SSP if and only if $\Psi$ has SSP. 
\label{5.10}
\end{theorem}
\begin{proof}
The proof is similar to the proof of Theorem \ref{3.13}.
\end{proof}
\begin{theorem}
Let $\Phi\in Act(G, X)$. If $x R^S_{\delta} y$, then $x WR^{S}_{\delta} y$. 
\label{5.5} 
\end{theorem}
\begin{proof}
Let $f$ be $\delta$-pseudo orbit such that $f(g) = x$ and $f(h) = y$ for distinct $g,h\in G$. Assume that $g = g_{1}g_{2}$...$g_{p}$ and $h = h_{1}h_{2}$...$h_{q}$, where $g_{i}, h_{j}\in S$ for all $1\leq i\leq p$, $1\leq j\leq q$. Since $d(\Phi_{s}(f(g)), f(sg)) < \delta$ for all $s\in S$, $g\in G$, the sequence $ x_{0}= x =f(g)$, $x_{1} = f(g_{2}$...$g_{p})$, $x_{2} = f(g_{3}$...$g_{p})$,..., $x_{p-1} = f(g_{p})$, $x_{p} = f(e)$, $x_{p+1} = f(h_{q})$, $x_{p+2} = f(h_{q-1}h(q))$,..., $x_{p+q} = f(h_{1}$...$h_{q})$ forms a weak $\delta$-chain via sequence $\lbrace g_{1}^{-1}, g_{2}^{-1}$,..., $g_{p}^{-1}, h_{q}, h_{q-1}$,...,$h_{1}\rbrace$ of elements of $S$ from $x$ to $y$. Similarly, we can construct a weak $\delta$-chain from $y$ to $x$. Hence, $x WR^S_{\delta} y$.  
\end{proof}

\begin{Corollary}
For $\Phi\in Act(G, X)$, transitivity implies chain transitivity implies weak chain transitivity. 
\label{5.6} 
\end{Corollary}
\begin{proof}
It follows from Theorem \ref{4.7}, Theorem \ref{5.5} together with \textit{Proposition 2.8} \cite{B1}. 
\end{proof}

\begin{Lemma}
Let $\Phi\in Act(G, X)$, then $B^{S}_{\lambda}$ is closed and $\Phi$-invariant for each $\lambda \in \Lambda$.
\label{5.7}
\end{Lemma}
\begin{proof}
Let $\epsilon > 0$ and $\lbrace y_{j}\rbrace$ be a sequence in $B^{S}_{\lambda} = [z]$ such that $y_{j}$ converges to $y$. 
Choose $m >0$ such that $d(y,y_{j})<\epsilon$ for all $j\geq m$. Since $y_{m}\in B^{S}_{\lambda}$, there exist weak $\epsilon$-chains $\lbrace y_{m} = x_{0}, x_{1}$,...,$x_{k-1}, x_{k} = z\rbrace$ and $\lbrace z = x_{0}, x_{1}$,...,$x_{k'-1}, x_{k'}=y_{m}\rbrace$. Fix $s\in S$, then $\lbrace y, \Phi_{s}(y), y_{m}, x_{1}$,...,$x_{k-1}, x_{k}=z\rbrace$ and $\lbrace z = x_{0}, x_{1}$, ...,$x_{k'-1}, x_{k'} = y_{m}, \Phi_{s}(y_{m}), y\rbrace$ are weak $\epsilon$-chains from $y$ to $z$ and from $z$ to $y$. Hence, $y\in B^{S}_{\lambda}$ and thus $B^{S}_{\lambda}$ is closed for each $\lambda\in \Lambda$.  

Now, Let $s'\in S$, $\lambda\in \Lambda$, $\epsilon > 0$. Suppose that $y\in \Phi_{s'}(B^{S}_{\lambda}) = \Phi_{s'}[z]$ for some $z\in X$. Then, $y = \Phi_{s'}(x)$ for some $x\in B^{S}_{\lambda}$. Let $\lbrace x = x_{0}, x_{1}$,...,$x_{k-1}, x_{k} = z\rbrace$ and $\lbrace z = x_{0}, x_{1}$,...,$x_{k'-1}, x_{k'} = x\rbrace$ are weak $\epsilon$-chains from $x$ to $z$ and from $z$ to $x$. Then, $\lbrace y = \Phi_{s'}(x), x = x_{0}, x_{1}$,...,$x_{k-1}, x_{k} = z\rbrace$ and $\lbrace z = x_{0}, x_{1}$,...,$x_{k'-1}, x_{k'} = x, \Phi_{s'}(x) \rbrace$ are weak $\epsilon$-chains from $y$ to $z$ and from $z$ to $y$. Hence, $\Phi_{s'}(B^{S}_{\lambda})\subset B^{S}_{\lambda}$. Since $s'$ and $\lambda$ was chosen arbitrary and $S$ is symmetric, we have  $B^{S}_{\lambda}\subset \Phi_{s'}(B^{S}_{\lambda})$. Thus, $B^{S}_{\lambda}$ is $\Phi$-invariant for every $\lambda\in \Lambda$. 
\end{proof}

\begin{theorem}
If $\Phi\in Act(G, X)$ has SSP, then $X = \Omega(\Phi)$. 
\label{5.11} 
\end{theorem}
\begin{proof}
Let $\epsilon > 0$ and let $\delta>0$ be given for $\epsilon$ by SSP of $\Phi$. Then, there exists weak $\delta$-chain $x_{0}=x, x_{1}$,...,$x_{k} = x$ with sequence $\lbrace s_{x_i}\rbrace_{i=0}^{k-1}$. We can extend this chain to $\delta$-SPO $\lbrace y_{i}\rbrace_{i=-\infty}^{\infty}$, where $y_{ki+j} = x_{j}$ for all $i\in \mathbb{Z}$, $0\leq j\leq k$ with sequence $\lbrace s_{x_i}\rbrace_{i=-\infty}^{\infty}$, where $s_{x_{ki+j}} = s_{x_j}$ for all $i\in \mathbb{Z}$, $0\leq j\leq k$. Thus, by SSP there exists $z\in X$ which $\epsilon$-ST this orbit and so we can choose non-identity element $g_{k}$ such that $d(\Phi_{g_{k}}(z), x)< \epsilon$. Hence, $\Phi_{g_{k}}(U_{\epsilon}(x))\cap U_{\epsilon}(x)\neq \phi$, where $U_{\epsilon}(x)=\lbrace y\in X : d(y,x) <\epsilon \rbrace$. Since $\epsilon$ was chosen arbitrary, we have $x\in \Omega(\Phi)$.
\end{proof}

\begin{Lemma} 
If $\Phi$ has SSP, then $B^{S}_{\lambda}$ is open in $\Omega(\Phi)$ for each $\lambda$.
\label{5.8} 
\end{Lemma}
\begin{proof}
We have $B^{S}_{\lambda}\subset \Omega(\Phi) = WCR^{S}(\Phi) = X$. Let us fix $x\in B^{S}_{\lambda}$ and $\delta > 0$. Choose $0 <\gamma <\delta$ such that $d(x,y)<\gamma$ implies $d(\Phi_{s}(x),\Phi_{s}(y))<\delta$ for all $s\in S$. For $y\in B^{S}_{\lambda}$ there is a $\delta$-chain $\lbrace x=x_{0},$...,$x_{k} = y\rbrace$ with sequence $\lbrace s_{x_{0}},$...,$s_{x_{k}}\rbrace$ in $\Omega(\Phi)$. It is enough to show that $U_{\gamma}(x)\subset B^{S}_{\lambda}$. Let $x'\in U_{\gamma}(x)$, then $\lbrace x', \Phi_{s_{x_{0}}}(x'), x_{0},$...,$x_{k}=y\rbrace$ is a $\delta$-chain with sequence $\lbrace s_{x_{0}}, s^{-1}_{x_{0}}, s_{x_{0}},$...,$s_{x_{k}}\rbrace$. Since $x, y\in \Omega(\Phi)$ there is a $\delta$-chain from $\lbrace y=y_{0},$...,$y_{l}=x\rbrace$. Hence, for any $s_{y_{l}}\in S$, we can construct a $\delta$-chain $\lbrace y=y_{0},$...,$y_{l}=x,\Phi_{s_{y_{l}}}(x'), x'\rbrace$.
\end{proof}

\begin{theorem}
If $\Phi\in Act(G,X)$ is weak chain transitive having SSP, then it is transitive.  
\label{5.12}
\end{theorem}
\begin{proof}
The proof is similar to the proof of Theorem \ref{5.11}.   
\end{proof}
\begin{Corollary}
If $\Phi\in Act(G, X)$ has SSP, then 
\begin{enumerate}
\item $\Omega(\Phi)= CR(\Phi)= X$.
\item $\Phi$ is weak chain transitive if and only if $\Phi$ is transitive.
\end{enumerate}
\label{5.13}
\end{Corollary}
\begin{proof}
Proof of (1) follows from Theorem \ref{5.11} and Lemma 2.10 \cite{B1} and proof of (2) follows from Theorem \ref{5.12} and Corollary \ref{5.6}.
\end{proof}

\begin{theorem}
Let $X$ be a connected metric space. If $\Phi\in Act(G,X)$ has SSP, then it is transitive. In addition, if $X$ is compact then transitivity of $\Phi$ implies that $\Phi$ has SSP.      
\label{5.14}
\end{theorem}
\begin{proof}
By Theorem \ref{5.11}, $X = \Omega(\Phi)$ which is closed and $\Phi$-invariant. We claim that $\Phi$ is weak chain transitive. Let $x$, $y$ be two distinct points in $X$ and $\epsilon > 0$. Choose $0 < \delta < \epsilon$ such that $d(a,b) < \delta$ implies that $d(\Phi_{s}(a), \Phi_{s}(b)) < \epsilon$ for all $s\in S$. By connectedness, choose a sequence of points $z_0 = x,$...$,z_k = y$ so that $d(z_{j+1},z_j) < \delta$ for all $j = 0,$...$,k-1$. 
Let $s'\in S$ be fixed and observe that the sequence $z_0 = x, \Phi_{s'}(z_{1}), z_{1}, \Phi_{s'}(z_{2}), z_{2}$...$z_{k-2}, \Phi_{s'}(z_{k-1}), z_{k-1}, \Phi_{s'}(z_{k}), z_k = y$ forms weak $\epsilon$-chain from $x$ to $y$. Similarly, we can form a weak $\epsilon$-chain from $y$ to $x$. Since $\epsilon$ was chosen arbitrary, we get that $x WR y$. Thus, $\Phi$ is weak chain transitive and hence by Theorem \ref{5.12} $\Phi$ is transitive. The additional fact holds because transitivity implies existence of dense orbit.   
\end{proof}


\begin{theorem}
Let $X$ be a compact metric space. If $\Phi\in Act(G,X)$ has SSP, then $X$ can be written as finite union of disjoint closed $\Phi$-invariant sets on each of which $\Phi$ is transitive. In addition, if $X$ is connected then the decomposition is trivial.    
\label{5.15}
\end{theorem}
\begin{proof} By Theorem \ref{5.11}, Lemma \ref{5.8} and Lemma \ref{5.7}, $X = \cup_{i=1}^{p}B^{S}_{\lambda_{i}} = \Omega(\Phi)$ where $B^{S}_{\lambda_{i}}$ are closed and invariant, for all $1\leq i\leq p$. 
For given $0 < \epsilon < min \lbrace d(B^{S}_{\lambda_{i}}, B^{S}_{\lambda_{j}}) : 1\leq i, j\leq p\rbrace$,  choose $ 0 < \delta < \epsilon$ by the SSP of $\Phi|_{\Omega(\Phi)} = \Phi$. Note that every $\delta$-SPO through $\Phi|_{B^{S}_{\lambda_{i}}}$ will be $\epsilon$-ST through a point of $B^{S}_{\lambda_{i}}$, for each $1\leq i\leq p$. Hence $\Phi|_{B^{S}_{\lambda_{i}}}$ has SSP, for each $1\leq i\leq p$. Since $\Phi|_{B^{S}_{\lambda_{i}}}$ is weakly chain transitive with respect to $S$, Theorem \ref{5.12} completes a proof of the first part. The additional part is clear from Theorem \ref{5.14}.  
\end{proof}

\begin{Example} Let $G = < a, b : ba =a^{n}b>$ be a finitely generated group, $\lambda > n > 1$ and 
\[
A =
\begin{bmatrix}
1 & 0 \\
1 & 1
\end{bmatrix},
B = 
\begin{bmatrix}
\lambda & 0 \\
0 & n\lambda
\end{bmatrix}
\]
Define $\Phi : G\times \mathbb{R}^{2} \rightarrow \mathbb{R}^{2}$ by $\Phi_{a}(x) =Ax$ and $\Phi_{b}(x) = Bx$ for all $x\in \mathbb{R}^{2}$. Note that any open disk of radius $\frac{p}{4}$ centred at point $(p, p)$, for $p > 0$ is moved vertically upward or downward under $\Phi_{a}$ and move only within the first quadrant without touching any axis under $\Phi_{b}$. Thus, under $\Phi$ such disk always move in first or fourth quadrant without touching $y$-axis. Therefore, $\Phi$ cannot be transitive. By Theorem 4.4 \cite{O2} $\Phi$ has the shadowing property and hence by Lemma 2.9 \cite{B1}, $\Phi$ is not chain transitive. 
\label{5.17}
\end{Example} 
\begin{Example} 
Let $G = < a, b : ba =a^{2}b>$ be a finitely generated group and $m>1$ be a fixed integer. Define $\Phi : G\times \mathbb{R}\rightarrow \mathbb{R}$ by $\Phi_{a}(x) = x$ and $\Phi_{b}(x) = mx$ for fixed $m > 1$. By Example \ref{E4.8}, $\Phi$ has the shadowing property. Since $\Phi$ is not transitive, therefore by Lemma 2.9 \cite{B1} it is not chain transitive. Clearly, $\Phi$ is weakly chain transitive. Thus, weak chain transitivity need not imply transitivity or chain transitivity. Further by Theorem \ref{5.12}, $\Phi$ does not have sequential shadowing property. Therefore, we conclude that the shadowing property need not imply sequential shadowing property.  
\label{5.16}
\end{Example} 

Acknowledgements: The first author is supported by CSIR-Junior Research Fellowship (File No.-09/045(1558)/2018-EMR-I) of  Government of India. The second author is supported by Department of Science and Technology, Government of India, under INSPIRE Fellowship (Resgistration No-IF150210) program.


\begin{thebibliography}{23}
\bibitem{A} D. V. Anosov, On a Class of Invariant Sets of Smooth Dynamical Systems, Proc. 5th Int. Conf. on Nonlin. Oscill., 2, Kiev (1970), 39-45.
\bibitem{A1} N. Aoki, K. Hiraide, Topological Theory of Dynamical Systems: Recent Advances, Elsevier, (1994). 
\bibitem{A21} N. Aoki, On Homeomorphisms with Pseudo-orbit Tracing Property, Tokyo J. Math., 6 (1983), 329-334.  
\bibitem{A3} A. Arbieto, C. A. Rojas, Topological Stability from Gromov-Hausdorff Viewpoint, Discrete \& Continuous Dynamical Systems-A, 37 (2017), 3531-3544. 
\bibitem{B} R. Bowen, Equilibrium States and the Ergodic Theory of Axiom-A Diffeomorphisms, Springer Lecture Notes, Vol. 470, (1975). 
\bibitem{B1} A. Barzanouni, Shadowing Property on Finitely Generated Group Actions, Journal Dynamical Systems and Geometric Theories, 12 (2014), 69-79.  
\bibitem{BDS} A. Barzanouni, M. S. Divandar, E. Shah, On Properties of Expansive Group Actions, To appear in Acta Mathematica Vietnamica.  
\bibitem{C3} N. P. Chung, K. Lee, Topological Stability and Pseudo-orbit Tracing Property of Group Actions, Proc. Amer. Math. Soc., 146 (2016).
\bibitem{DKY} M. Dong, S. Kim, J. Yin, Group Actions with Topologically Stable Measures, Dynamic Systems and Applications, 27 (2018), 185-199.   
\bibitem{H1} S. Hurder, Dynamics of Expansive Group Actions on the Circle, Preprint, Aug 27 (2000).
\bibitem{K1} H. B. Keynes, J. B. Robertson, Generators for Topological Entropy and Expansiveness, Mathematical Systems Theory, 3 (1969), 51-59.
\bibitem{LM} K. Lee, C. A. Morales, Topological Stability and Pseudo Orbit Tracing Property for Expansive Measures, J. Differential Equations, 262 (2017), 3467-3487.
\bibitem{M} C. A. Morales, Measure-expansive Systems, Preprint, IMPA, D083 (2011).    
\bibitem{N} Z. Nitecki, Differentiable Dynamics, M.I.T. Press, (1971).   
\bibitem{O2} A. V. Osipov, S. B. Tikhomirov, Shadowing for Actions of Some Finitely Generated Groups, Dynamical Systems, 29 (2014), 337-351.
\bibitem{P1} S. Y. Pilyugin, S. B. Tikhomirov, Shadowing in Actions of Some Abelian Groups, Fundamenta Mathematicae, 179 (2003), 83-96.   
\bibitem{S1} S. Smale, Differentiable Dynamical Systems, Bull. Amer. Math. Soc., 73 (1967), 747-817.
\bibitem{U1} W. R. Utz, Unstable homeomorphisms, Proc. Amer. Math. Soc., 1 (1950), 769-774.  
\bibitem{U2} P. Walters, Anosov Diffeomorphisms are Topologically Stable, Topology 9 (1970), 71-78.
\bibitem{U3} P. Walters, On the Pseudo Orbit Tracing Property and its Relationship to Stability, Lecture Notes
in Math., vol. 668, Springer, Berlin, 1978, 231-244.
\end{thebibliography}
\end{document}